\documentclass[11pt]{article}
\usepackage{amsmath}
\usepackage{amsfonts}
\usepackage{amssymb}
\usepackage{mathrsfs}
\usepackage{amsthm}
\usepackage{palatino}
\usepackage[margin=2.5cm, vmargin={1.5cm}]{geometry}

\usepackage{dcolumn}

\usepackage{times}
\usepackage{graphicx}
\usepackage{xcolor}

\usepackage{pstricks}
\usepackage{pst-plot}
\usepackage{pst-grad}

\newrgbcolor{lightblue}{0.8 0.8 1.0}

\renewcommand\footnotemark{}

\begin{document}

\title{Ramanujan's approximation to the exponential function and generalizations}

\author{
  Cormac ~O'Sullivan\footnote{
\newline \indent \ \ \
  \newline \indent \ \ \
Support for this project was provided by a PSC-CUNY Award, jointly funded by The Professional Staff Congress and The City
\newline \indent \ \ \
University of New York.}
  }

\date{}

\maketitle

\def\s#1#2{\langle \,#1 , #2 \,\rangle}

\def\F{{\frak F}}
\def\C{{\mathbb C}}
\def\R{{\mathbb R}}
\def\Z{{\mathbb Z}}
\def\Q{{\mathbb Q}}
\def\N{{\mathbb N}}
\def\G{{\Gamma}}
\def\GH{{\G \backslash \H}}
\def\g{{\gamma}}
\def\L{{\Lambda}}
\def\ee{{\varepsilon}}
\def\K{{\mathcal K}}
\def\Re{\mathrm{Re}}
\def\Im{\mathrm{Im}}
\def\PSL{\mathrm{PSL}}
\def\SL{\mathrm{SL}}
\def\Vol{\operatorname{Vol}}
\def\lqs{\leqslant}
\def\gqs{\geqslant}
\def\sgn{\operatorname{sgn}}
\def\res{\operatornamewithlimits{Res}}
\def\li{\operatorname{Li_2}}
\def\lip{\operatorname{Li}'_2}
\def\pl{\operatorname{Li}}

\def\ei{\mathrm{Ei}}

\def\clp{\operatorname{Cl}'_2}
\def\clpp{\operatorname{Cl}''_2}
\def\farey{\mathscr F}

\def\dm{{\mathcal A}}
\def\ov{{\overline{p}}}
\def\ja{{K}}

\def\nb{{\mathcal B}}
\def\cc{{\mathcal C}}
\def\nd{{\mathcal D}}

\newcommand{\stira}[2]{{\genfrac{[}{]}{0pt}{}{#1}{#2}}}
\newcommand{\stirb}[2]{{\genfrac{\{}{\}}{0pt}{}{#1}{#2}}}
\newcommand{\eu}[2]{{\left\langle\!\! \genfrac{\langle}{\rangle}{0pt}{}{#1}{#2}\!\!\right\rangle}}
\newcommand{\eud}[2]{{\left\langle\! \genfrac{\langle}{\rangle}{0pt}{}{#1}{#2}\!\right\rangle}}
\newcommand{\norm}[1]{\left\lVert #1 \right\rVert}

\newcommand{\e}{\eqref}
\newcommand{\bo}[1]{O\left( #1 \right)}


\newtheorem{theorem}{Theorem}[section]
\newtheorem{lemma}[theorem]{Lemma}
\newtheorem{prop}[theorem]{Proposition}
\newtheorem{conj}[theorem]{Conjecture}
\newtheorem{cor}[theorem]{Corollary}
\newtheorem{assume}[theorem]{Assumptions}
\newtheorem{adef}[theorem]{Definition}


\newcounter{counrem}
\newtheorem{remark}[counrem]{Remark}

\renewcommand{\labelenumi}{(\roman{enumi})}
\newcommand{\spr}[2]{\sideset{}{_{#2}^{-1}}{\textstyle \prod}({#1})}
\newcommand{\spn}[2]{\sideset{}{_{#2}}{\textstyle \prod}({#1})}

\numberwithin{equation}{section}

\let\originalleft\left
\let\originalright\right
\renewcommand{\left}{\mathopen{}\mathclose\bgroup\originalleft}
\renewcommand{\right}{\aftergroup\egroup\originalright}

\bibliographystyle{alpha}

\begin{abstract}
Ramanujan's approximation to the exponential function is reexamined with the help of Perron's saddle-point method. This allows for a wide generalization that includes the results of Buckholtz, and where all the  asymptotic expansion coefficients may be given in closed form. Ramanujan's approximation to the exponential integral is treated similarly.
\end{abstract}

\section{Introduction}
\subsection{Ramanujan's approximation to $e^n$}
The largest terms in the Taylor series development of $e^n$, when $n$ is  a positive integer, are $n^j/j!$ for $j=n-1$ and $j=n$.
So it is natural to compare $e^n/2$  with the sum of the first $n$ terms of this series. Ramanujan did this in Entry 48 of Chapter 12 in his second notebook,  writing
\begin{equation}\label{rb}
  1+\frac{n}{1!}+\frac{n^2}{2!}+ \cdots +\frac{n^{n-1}}{(n-1)!} +  \frac{n^n}{n!}\theta_n = \frac{e^n}2,
\end{equation}
and computing an asymptotic expansion which is equivalent to
\begin{equation}\label{rb2}
   \theta_n  = \frac{1}3+ \frac{4}{135 n}-\frac{8}{2835 n^2}-\frac{16}{8505 n^3}+\frac{8992}{12629925 n^4}+O\left( \frac{1}{n^5}\right),
\end{equation}
as $n \to \infty$.
Label the coefficient of $n^{-r}$ in the above  expansion as $\rho_r$. The difficulty of computing  $\rho_r$ in general was resolved by Marsaglia in \cite{Mar86} with a recursive procedure. In this paper, all expansion coefficients are given in closed form.
For example, one of our formulas for $\rho_r$  
is
\begin{equation*}
  \rho_r   = -  \sum_{k=0}^{2r+1}  \frac{(2r+2k)!!}{(-1)^k k!} \dm_{2r+1,k}\left( \frac{1}{3!}, \frac{1}{4!}, \frac{1}{5!}, \dots \right),
\end{equation*}
where the De Moivre polynomials $\dm_{n,k}$ are described in Definition \ref{dbf}. An elementary formula  for the quantity $\dm_{m,k}\left(\frac{1}{3!}, \frac{1}{4!}, \frac{1}{5!}, \dots \right)$ is  shown in \e{mvw2}. The usual double factorial notation we are using  has
\begin{equation} \label{doub}
  n!! := \begin{cases}
  n(n-2) \cdots 5 \cdot 3 \cdot 1 & \text{ if $n$ is odd};\\
  n(n-2) \cdots 6 \cdot 4 \cdot 2 & \text{ if $n$ is even},
  \end{cases}
\end{equation}
for $n\gqs 1$, with $0!!=(-1)!!=1$.


Inspired by claims of Ramanujan,  Szeg\"o in 1928 and Watson in 1929 bounded $\theta_n$ from above and below.  Flajolet and coauthors in 1995 \cite[Sect. 1]{Fl95} established a finer estimate and this result was elegantly reproved and extended by Volkmer  \cite{Vo08}, employing the Lambert $W$ function. See the discussion of  much work related to Entry 48 in \cite[pp. 181--184]{Ber89}.

Equation \e{rb} may be generalized to summing the first $n+v$ terms of the series:
\begin{equation}\label{rb3}
  \sum_{j=0}^{n+v-1}\frac{n^j}{j!}+   \frac{n^{n+v}}{(n+v)!}\theta_n(v) = \frac{e^n}2.
\end{equation}
Ramanujan developed the asymptotics of a related integral, (see \e{vv} with $w=1$), as described in  \cite[p. 193]{Ber89}, and his result is equivalent to
\begin{multline}\label{hrb}
   \theta_n(v)  = \frac{1}3-v+ \left(\frac{4}{135}-\frac{v^2(v+1)}{3}\right)\frac 1{n}\\
   -
   \left(\frac{8}{2835}+\frac{v(9v^4-15v^2-2v+4)}{135}\right)\frac{1}{n^2}+O\left( \frac{1}{n^3}\right),
\end{multline}
for $v$ a fixed integer as $n \to \infty$.
Our version of \e{hrb} is given in Theorem \ref{ext}.

\subsection{Further asymptotics}


We may also  include another natural parameter $w$.
Let $n$ and $v$ be integers with $n \gqs 1$ and $n+v \gqs 0$. For  nonzero $w\in \C$, define $S_n(w;v)$ with
\begin{equation}\label{ss}
  e^{n w}  =\sum_{j=0}^{n+v-1} \frac{(n w)^j}{j!}
   + \frac{(n w)^{n+v}}{(n+v)!} S_n(w;v)
\end{equation}
and define the complimentary $T_n(w;v)$ with
\begin{equation}\label{tt}
  e^{n w} = \frac{(n w)^{n+v}}{(n+v)!} T_n(w;v)
   + \sum_{j=n+v}^{\infty} \frac{(n w)^j}{j!}.
\end{equation}
For $w=0$ we may  set $S_n(0;v)$ to be $0$, leaving $T_n(0;v)$ undefined. It is clear from \e{tt} that $T_n(w;v)$ can  be given as a finite sum:
\begin{equation}\label{tt2}
   \frac{(n w)^{n+v}}{(n+v)!} T_n(w;v)
   = \sum_{j=0}^{n+v-1} \frac{(n w)^j}{j!}.
\end{equation}
Also there is the relation
\begin{equation}\label{tt3}
  e^{n w} = \frac{(n w)^{n+v}}{(n+v)!}\left(S_n(w;v)+ T_n(w;v)\right).
\end{equation}
Our earlier $\theta_n(v)$ function from \e{rb3} occurs in the $w=1$ case:
\begin{align}\label{ori}
  \theta_n(v) & = \frac{S_n(1;v)}2 - \frac{T_n(1;v)}2
  \\
  & = S_n(1;v) -\frac{(n+v)!}{2 n^{n+v}} e^n =  \frac{(n+v)!}{2 n^{n+v}} e^n -T_n(1;v). \label{ori2}
\end{align}

The behavior of $S_n(w;v)$ and $T_n(w;v)$ as  $n \to \infty$ are our main results in Theorems \ref{msth} and \ref{mtth}, extending the $v=0$ case considered by Buckholtz in \cite{Bu63}. These asymptotics depend on which of certain regions $w$ lies in; see Figure \ref{fig}. The region $\mathcal X$ is given by $\{w\in \C\,:\, |w e^{1-w}|>1\}$. Also $\{w\in \C\,:\, |w e^{1-w}|<1\}$ has the disjoint parts $\mathcal Y$ with $\Re(w)<1$ and $\mathcal Z$ with $\Re(w)>1$. The boundary curves $\mathcal S$ and $\mathcal T$ are where $|w e^{1-w}|=1$ with $\Re(w)<1$ and  $\Re(w)>1$, respectively. These curves have the parametrization $t\pm i \sqrt{e^{2t-2}-t^2}$ for $t\gqs -W(1/e) \approx -0.2785$, using the Lambert $W$ function. Among other things, Szeg\"o showed in \cite{Sz24} that $z$ is an accumulation point for the zeros of $T_n(w;1)$ as $n\to \infty$ if and only if $z \in \mathcal S \cup\{1\}$; this is the Szeg\"o curve.

%

\SpecialCoor
\psset{griddots=5,subgriddiv=0,gridlabels=0pt}
\psset{xunit=3cm, yunit=3cm}
\psset{linewidth=1pt}
\psset{dotsize=4pt 0,dotstyle=*}

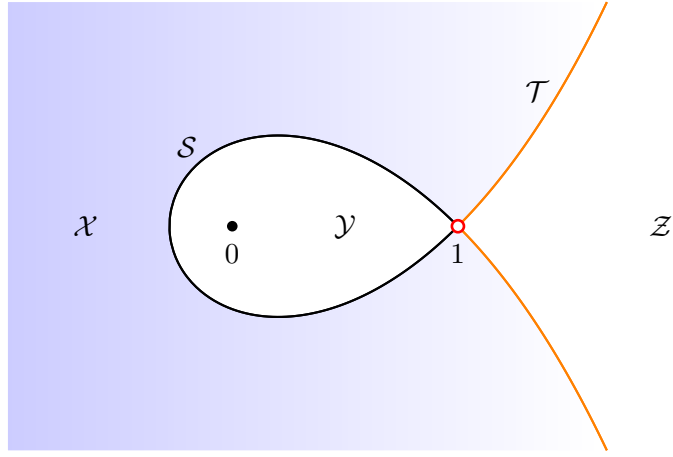
\begin{figure}[ht]
\begin{center}
\begin{pspicture}(-1,-1)(2,1) 

\pscustom[fillstyle=gradient,gradangle=270,linecolor=white,gradmidpoint=1,gradbegin=white,gradend=lightblue,gradlines=100]{%
  \pspolygon[linecolor=lightblue](-1,-1)(1.66,-1)(1.66,1)(-1,1)(-1,1)} 

\savedata{\mydatay}[
{{1., 0.}, {0.98, 0.0197342}, {0.96, 0.0389403}, {0.94,
  0.0576232}, {0.92, 0.0757878}, {0.9, 0.0934385}, {0.88,
  0.11058}, {0.86, 0.127215}, {0.84, 0.143349}, {0.82,
  0.158985}, {0.8, 0.174127}, {0.78, 0.188776}, {0.76,
  0.202937}, {0.74, 0.216612}, {0.72, 0.229802}, {0.7,
  0.242511}, {0.68, 0.25474}, {0.66, 0.26649}, {0.64,
  0.277763}, {0.62, 0.288559}, {0.6, 0.29888}, {0.58,
  0.308724}, {0.56, 0.318093}, {0.54, 0.326985}, {0.52, 0.3354}, {0.5,
   0.343336}, {0.48, 0.350792}, {0.46, 0.357765}, {0.44,
  0.364252}, {0.42, 0.370252}, {0.4, 0.375758}, {0.38,
  0.380768}, {0.36, 0.385276}, {0.34, 0.389275}, {0.32,
  0.39276}, {0.3, 0.395723}, {0.28, 0.398155}, {0.26,
  0.400047}, {0.24, 0.401387}, {0.22, 0.402164}, {0.2,
  0.402364}, {0.18, 0.40197}, {0.16, 0.400966}, {0.14,
  0.399332}, {0.12, 0.397045}, {0.1, 0.39408}, {0.08,
  0.390407}, {0.06, 0.385992}, {0.04, 0.380798}, {0.02,
  0.374778}, {0., 0.367879}, {-0.02, 0.36004}, {-0.04,
  0.351184}, {-0.06, 0.341221}, {-0.08, 0.330038}, {-0.1,
  0.317495}, {-0.12, 0.303411}, {-0.14, 0.287549}, {-0.16,
  0.26958}, {-0.18, 0.249039}, {-0.2, 0.225206}, {-0.2,
  0.225206}, {-0.21, 0.211711}, {-0.22, 0.196878}, {-0.23,
  0.180374}, {-0.24, 0.161689}, {-0.25, 0.139946}, {-0.26,
  0.1134},
  {-0.27, 0.0772425}, {-0.271, 0.0725798}, {-0.272, 0.0675838},
{-0.273, 0.0621741}, {-0.274, 0.0562315}, {-0.275, 0.0495648},
{-0.276, 0.0418289}, {-0.277, 0.032264},
{-0.278, -0.0181818}, {-0.277, -0.032264}, {-0.276, -0.0418289},
{-0.275, -0.0495648}, {-0.274, -0.0562315}, {-0.273, -0.0621741},
{-0.272, -0.0675838}, {-0.271, -0.0725798},
  {-0.27, -0.0772425}, {-0.26, -0.1134}, {-0.25,
-0.139946}, {-0.24, -0.161689}, {-0.23, -0.180374}, {-0.22,
-0.196878}, {-0.21, -0.211711}, {-0.2, -0.225206}, {-0.2, -0.225206},
{-0.18, -0.249039}, {-0.16, -0.26958}, {-0.14, -0.287549}, {-0.12,
-0.303411}, {-0.1, -0.317495}, {-0.08, -0.330038}, {-0.06,
-0.341221}, {-0.04, -0.351184}, {-0.02, -0.36004}, {0., -0.367879},
{0.02, -0.374778}, {0.04, -0.380798}, {0.06, -0.385992}, {0.08,
-0.390407}, {0.1, -0.39408}, {0.12, -0.397045}, {0.14, -0.399332},
{0.16, -0.400966}, {0.18, -0.40197}, {0.2, -0.402364}, {0.22,
-0.402164}, {0.24, -0.401387}, {0.26, -0.400047}, {0.28, -0.398155},
{0.3, -0.395723}, {0.32, -0.39276}, {0.34, -0.389275}, {0.36,
-0.385276}, {0.38, -0.380768}, {0.4, -0.375758}, {0.42, -0.370252},
{0.44, -0.364252}, {0.46, -0.357765}, {0.48, -0.350792}, {0.5,
-0.343336}, {0.52, -0.3354}, {0.54, -0.326985}, {0.56, -0.318093},
{0.58, -0.308724}, {0.6, -0.29888}, {0.62, -0.288559}, {0.64,
-0.277763}, {0.66, -0.26649}, {0.68, -0.25474}, {0.7, -0.242511},
{0.72, -0.229802}, {0.74, -0.216612}, {0.76, -0.202937}, {0.78,
-0.188776}, {0.8, -0.174127}, {0.82, -0.158985}, {0.84, -0.143349},
{0.86, -0.127215}, {0.88, -0.11058}, {0.9, -0.0934385}, {0.92,
-0.0757878}, {0.94, -0.0576232}, {0.96, -0.0389403}, {0.98,
-0.0197342}, {1., 0.}}
]
\dataplot[linecolor=black,linewidth=0.8pt,plotstyle=line,fillstyle=solid,fillcolor=white]{\mydatay}

\savedata{\mydata}[
{{1.66, 0.993892}, {1.64, 0.952386}, {1.62, 0.911709}, {1.6,
  0.871847}, {1.58, 0.832786}, {1.56, 0.794515}, {1.54,
  0.75702}, {1.52, 0.72029}, {1.5, 0.684311}, {1.48, 0.649074}, {1.46,
   0.614565}, {1.44, 0.580775}, {1.42, 0.547692}, {1.4,
  0.515307}, {1.38, 0.483608}, {1.36, 0.452585}, {1.34,
  0.422229}, {1.32, 0.392531}, {1.3, 0.363481}, {1.28,
  0.335071}, {1.26, 0.307291}, {1.24, 0.280133}, {1.22,
  0.253589}, {1.2, 0.22765}, {1.18, 0.20231}, {1.16, 0.177561}, {1.14,
   0.153394}, {1.12, 0.129804}, {1.1, 0.106784}, {1.08,
  0.084326}, {1.06, 0.0624248}, {1.04, 0.0410739}, {1.02,
  0.0202676}, {1., 0.}, {1.,
  0.}, {1.02, -0.0202676}, {1.04, -0.0410739}, {1.06, -0.0624248},
{1.08, -0.084326}, {1.1, -0.106784}, {1.12, -0.129804}, {1.14,
-0.153394}, {1.16, -0.177561}, {1.18, -0.20231}, {1.2, -0.22765},
{1.22, -0.253589}, {1.24, -0.280133}, {1.26, -0.307291}, {1.28,
-0.335071}, {1.3, -0.363481}, {1.32, -0.392531}, {1.34, -0.422229},
{1.36, -0.452585}, {1.38, -0.483608}, {1.4, -0.515307}, {1.42,
-0.547692}, {1.44, -0.580775}, {1.46, -0.614565}, {1.48, -0.649074},
{1.5, -0.684311}, {1.52, -0.72029}, {1.54, -0.75702}, {1.56,
-0.794515}, {1.58, -0.832786}, {1.6, -0.871847}, {1.62, -0.911709},
{1.64, -0.952386}, {1.66, -0.993892}}
]
\dataplot[linecolor=orange,linewidth=0.8pt,plotstyle=line,fillstyle=solid,fillcolor=white]{\mydata}

\dataplot[linecolor=black,linewidth=0.8pt,plotstyle=line]{\mydatay}
\dataplot[linecolor=orange,linewidth=0.8pt,plotstyle=line]{\mydata}

\pscircle*[linecolor=white,linewidth=1pt](1,0){0.08}
\pscircle[linecolor=red,linewidth=1pt](1,0){0.08}

\psdot(0,0)
\rput(1,-0.12){$1$}
\rput(0,-0.12){$0$}

\rput(-0.65,0.0){$\mathcal X$}
\rput(0.5,0.0){$\mathcal Y$}
\rput(1.9,0.0){$\mathcal Z$}

\rput(-0.2,0.35){$\mathcal S$}
\rput(1.35,0.6){$\mathcal T$}

\end{pspicture}
\caption{Partitioning the $w$-plane into $\mathcal X \cup \mathcal Y \cup \mathcal Z  \cup \mathcal S  \cup \mathcal T \cup \{ 1\}$  \label{fig}}
\end{center}
\end{figure}


Perron's saddle-point method is reviewed in section \ref{saddle}, and all the asymptotic expansions in this paper are proved as applications of this theory.
Our work also naturally includes the following version of Stirling's approximation.

\begin{prop} \label{xpg}
Let $v$ be any complex number. As real $n \to \infty$,
\begin{equation} \label{gnv}
  \G(n+v+1)= \sqrt{2\pi n}\frac{n^{n+v}}{e^n} \left(1+\frac{\g_1(v)}{n}+\frac{\g_2(v)}{n^2}+ \cdots + \frac{\g_{R-1}(v)}{n^{R-1}} +O\left(\frac{1}{n^{R}}\right)\right),
\end{equation}
for an implied constant depending only on $R$ and $v$, with
\begin{equation}\label{gnv2}
  \g_r(v) = \sum_{m=0}^{2r} (-1)^m \binom{v}{2r-m}\sum_{k=0}^{m}  \frac{(2r+2k-1)!!}{(-1)^k k!} \dm_{m,k}\left( \frac{1}{3}, \frac{1}{4}, \frac{1}{5}, \dots \right).
\end{equation}
\end{prop}

 The same techniques are used in the final section to examine Ramanujan's approximation to the exponential integral $\ei(n)$, which may be defined as a Cauchy principal value:
\begin{equation} \label{ein}
  \ei(n):= \lim_{\varepsilon \to 0^+} \left(\int_{-\infty}^{-\varepsilon} \frac{e^t}t \, dt + \int_{\varepsilon}^{n} \frac{e^t}t \, dt\right).
\end{equation}
We describe there an unexplained connection between these approximations to $\ei(n)$ and $e^{n}$.


\section{De Moivre polynomials and the saddle-point method} \label{saddle}

\begin{adef} \label{dbf}
 For integers $n$, $k$ with $k\gqs 0$, the {\em De Moivre polynomial} $\dm_{n,k}(a_1, a_2, \dots)$ is defined by
 \begin{equation} \label{bell}
    \left( a_1 x +a_2 x^2+ a_3 x^3+ \cdots \right)^k = \sum_{n\in \Z} \dm_{n,k}(a_1, a_2, a_3, \dots) x^n \qquad \quad (k \in \Z_{\gqs 0}).
\end{equation}
\end{adef}

Many properties of these polynomials are assembled in \cite{odm}.
Clearly $\dm_{n,k}(a_1, a_2, a_3, \dots)=0$ if $n<k$. If $n\gqs k$ then
 \begin{equation} \label{bell2}
  \dm_{n,k}(a_1, a_2, a_3, \dots) = \sum_{\substack{1j_1+2 j_2+ \dots +mj_m= n \\ j_1+ j_2+ \dots +j_m= k}}
 \binom{k}{j_1 , j_2 ,  \dots , j_m} a_1^{j_1} a_2^{j_2}  \cdots a_m^{j_m},
\end{equation}
where $m=n-k+1$ and the sum   is over all possible $j_1$, $j_2$,  \dots , $j_m \in \Z_{\gqs 0}$. It is a polynomial in $a_1, a_2, \dots, a_{m}$ of homogeneous degree $k$ with positive integer coefficients.
 As in \cite[Sect. 2]{odm},
 we have the  relations
\begin{align}
\dm_{n,k}(0, a_1, a_2, a_3, \dots) & =  \dm_{n-k,k}(a_1, a_2, a_3, \dots), \label{gsb}\\
  \dm_{n,k}(c a_1, c a_2, c a_3, \dots) & = c^k \dm_{n,k}(a_1, a_2, a_3, \dots), \label{mulk}\\
  \dm_{n,k}(c a_1, c^2 a_2, c^3 a_3, \dots) & = c^n \dm_{n,k}(a_1, a_2, a_3, \dots). \label{muln}
\end{align}

In the paper \cite{OSper} we give a detailed description of Perron's saddle-point method from \cite{Pe17}.
 The main result requires the following assumptions and definitions.

\begin{assume} \label{ma0}
Let $\nb$ be a neighborhood of $z_0 \in \C$ and $\cc$   a contour  of integration containing $z_0$. Assume that $\cc$  lies in a bounded region of $\C$ and is parameterized by
  a continuous function   $c:[0,1]\to \C$ that has a continuous derivative except at a finite number of points. Suppose $p(z)$ and $q(z)$ are holomorphic functions on a domain containing $\nb \cup \cc$.  We  assume $p(z)$ is not constant and  hence there must exist $\mu \in \Z_{\gqs 1}$ and $p_0 \in \C_{\neq 0}$ so that
\begin{equation}
    p(z)  =p(z_0)-p_0(z-z_0)^\mu(1-\phi(z)) \qquad  (z\in \nb) \label{f}
\end{equation}
with $\phi$  holomorphic on $\nb$ and $\phi(z_0)=0$.
We will need the {\em steepest-descent angles}
\begin{equation}\label{bisec}
    \theta_\ell := -\frac{\arg(p_0)}{\mu}+\frac{2\pi \ell}{\mu} \qquad (\ell \in \Z).
\end{equation}
Assume that  $\nb,$ $\cc,$ $p(z),$  $q(z)$ and $z_0$    are independent of  $n>0$. Finally, let $K_q$ be a bound for $|q(z)|$ on $\nb \cup \cc$.
\end{assume}

\begin{theorem}\label{il} {\rm (Perron's method for a holomorphic integrand with  contour starting at a maximum.)}
Suppose that  Assumptions \ref{ma0} hold, with $\cc$  a contour from $z_0$ to $z_1$ in $\C$ where $z_0 \neq z_1$. Suppose  that
\begin{equation}\label{c1}
    \Re(p(z))<\Re(p(z_0)) \quad \text{for all} \quad z \in \cc, \ z\neq z_0.
\end{equation}
We may choose $k \in \Z$ so that the initial part of $\cc$ lies in the sector of angular width $2\pi/\mu$ about $z_0$ with bisecting angle $\theta_k$.
Then for every $S \in \Z_{\gqs 0}$, we have
\begin{equation} \label{wim}
    \int_\cc e^{n \cdot p(z)}  q(z) \, dz = e^{n \cdot p(z_0)} \left(\sum_{s=0}^{S-1}  \G\left(\frac{s+1}{\mu}\right) \frac{\alpha_s \cdot e^{2\pi i k (s+1)/\mu}}{n^{(s+1)/\mu}} + O\left(\frac{K_q}{n^{(S+1)/\mu}} \right)  \right)
\end{equation}
as $n \to \infty$ where  the implied constant in \eqref{wim} is independent of $n$ and $q$. The numbers $\alpha_s$
depend only on $s$, $p$, $q$ and $z_0$.
\end{theorem}

Theorem \ref{il} is \cite[Thm. 1.2]{OSper} and the next proposition  is \cite[Prop. 7.2]{OSper}.
Write the Taylor expansions of $p$ and $q$ at $z_0$ as
\begin{equation} \label{pqe}
    p(z)-p(z_0)=-\sum_{s=0}^\infty p_s (z-z_0)^{s+\mu}, \qquad q(z)=\sum_{s=0}^\infty q_s (z-z_0)^{s}.
\end{equation}

\begin{prop} \label{wojf}
The numbers needed in Theorem \ref{il} have the explicit formula
\begin{equation} \label{hjw}
    \alpha_s = \frac{1}{\mu} p_0^{-(s+a)/\mu} \sum_{m=0}^s q_{s-m}\sum_{j=0}^m  \binom{-(s+a)/\mu}{j}  \dm_{m,j}\left(\frac{p_1}{p_0},\frac{p_2}{p_0},\cdots\right),
\end{equation}
for $a=1$. (We will need an extension of this later, requiring more general $a$ values.)
\end{prop}

\section{Initial results for $S_n(w;v)$ and $T_n(w;v)$}

Define
\begin{equation}\label{pz}
  p(z)=p(w;z) := w(1-z)+\log z.
\end{equation}
 For $w$, $v$  in $\C$ with $w \neq 0$ and positive real $n$, we will need the function
\begin{equation}\label{pr}
  \frac{1}{(n w)^{n+v}} = e^{-(n+v) \log(nw)} = e^{-(n+v) \log n} \cdot e^{-(n+v) \log w},
\end{equation}
where $\log w$ is evaluated using the principle branch of the logarithm with arguments in $(-\pi,\pi]$.

\begin{lemma} \label{snwx}
The following formulas can be used to extend the definitions \e{ss} and \e{tt} of $S_n(w;v)$ and $T_n(w;v)$  to  all $w$, $v$  in $\C$ and $n$ in $\R$ with $n>0$ and $\Re(n+v)>-1$:
\begin{align} \label{snw2}
    S_n(w;v) & =1+n w  \int_0^1 e^{n \cdot p(z)} z^v \, dz,\\
    T_n(w;v) & =\frac{e^{n w}}{(n w)^{n+v}} \G(n+v+1) - S_n(w;v) \qquad (w\neq 0) \label{snw3}.
\end{align}
As functions of $w$, $S_n(w;v)$ is entire and $T_n(w;v)$ is holomorphic outside $(-\infty,0]$.
\end{lemma}
\begin{proof}
We first assume that $n$ and $v$ are integers with $n\gqs 1$ and $n+v \gqs 0$. Then for nonzero $w\in \C$ the integral
\begin{equation}\label{vv}
  \int_0^\infty e^{-t}\left(1+\frac t{n w} \right)^{n+v} \, dt
\end{equation}
is absolutely convergent and by the binomial theorem it
equals
\begin{equation*}
  \sum_{j=0}^{n+v} \binom{n+v}{j} \int_0^\infty e^{-t} \left(\frac t{n w} \right)^j \, dt
  =\frac{(n+v)!}{(n w)^{n+v}} \sum_{j=0}^{n+v} \frac{(n w)^j}{j!}.
\end{equation*}
Hence \e{tt2} implies that \e{vv} equals $1+T_n(w;v)$. With a change of variables,
\begin{equation} \label{bri}
   1+T_n(w;v) = n w  \int_{\nd} e^{-n w z} (1+z)^{n+v} \, dz = n w  \int_{1+\nd} e^{n w(1-z)} z^{n+v} \, dz
\end{equation}
for $\nd = \nd_w$  the line from $0$ through $1/w$ to infinity.
Next,
\begin{align}
  \frac{e^{n w}}{(n w)^{n+v}}(n+v)!  & = \frac{e^{n w}}{(n w)^{n+v}} \int_0^\infty e^{-t} t^{n+v} \, dt \notag\\
  & = \frac{e^{n w}}{(n w)^{n+v}} n w\int_{\nd} e^{-nw z} (n w z)^{n+v} \, dz \notag \\
  & =  n w\int_{\nd} e^{nw (1-z)} z^{n+v} \, dz. \label{bri2}
\end{align}
From \e{tt3}, \e{bri} and \e{bri2}, writing $\G(n+v+1)$ for $(n+v)!$,
\begin{align}
   S_n(w;v) & =\frac{e^{n w}}{(n w)^{n+v}} \G(n+v+1) - T_n(w;v) \label{bri3}\\
   & = 1+n w\int_{\nd} e^{nw (1-z)} z^{n+v} \, dz
   -n w  \int_{1+\nd} e^{n w(1-z)} z^{n+v} \, dz. \notag
\end{align}
Integrating this  holomorphic integrand around a closed contour gives zero, and so a limiting argument implies  \e{snw2} for nonzero $w\in \C$ and integers $n$, $v$  with $n\gqs 1$ and $n+v \gqs 0$. Now the integral in \e{snw2} converges to an entire function of $w$ for all $n$, $v$ with $\Re(n+v)>-1$, extending the definition of $S_n(w;v)$. Also \e{snw3}
follows from \e{bri3}, allowing the definition of $T_n(w;v)$ to be extended.
\end{proof}

\begin{lemma} \label{xm}
Let $w$, $v$ be in $\C$. As real $n \to \infty$,
\begin{alignat}{2} \label{lmf}
    S_n(w;v) & = 1+n w  \int_{1/2}^1 e^{n \cdot p(z)} z^v \, dz +O\left(2^{-n/20}\right) \qquad & &(\Re(w)\lqs 1),\\
    T_n(w;v) & = -1+n w  \int_{1}^{3/2} e^{n \cdot p(z)} z^v \, dz  +O\left(e^{-n/30}\right) \qquad & &(\Re(w)\gqs 1), \label{lmf2}
\end{alignat}
for implied constants depending only on $w$ and $v$.
\end{lemma}
\begin{proof}
With Lemma \ref{snwx}, to demonstrate \e{lmf} we must bound
\begin{equation} \label{dan}
  n w  \int_0^{1/2} e^{n \cdot p(z)} z^v \, dz \ll  n   \int_0^{1/2}\left( e^{\Re(w)(1-z)}z \right)^n z^{\Re(v)} \, dz.
\end{equation}
Use the inequality $e^{1-z}z \lqs z^{1/10}$ for $0\lqs z \lqs 1/2$ to continue, with
\begin{equation*}
   n   \int_0^{1/2}z^{n/10+ \Re(v)} \, dz \ll n   \int_0^{1/2}z^{n/20} \, dz,
\end{equation*}
for $n$ large enough that $n/20+ \Re(v)\gqs 0$, and we obtain \e{lmf}.

For \e{lmf2} we claim first that when $\Re(w)\gqs 1$ and $n$ is large enough,
\begin{equation}\label{ct}
   T_n(w;v)  = -1+n w  \int_{1}^{\infty} e^{n \cdot p(z)} z^v \, dz,
\end{equation}
and by \e{snw3} this is true if we can establish
\begin{equation}\label{ct2}
   \frac{e^{n w}}{(n w)^{n+v}} \G(n+v+1) = n w  \int_0^\infty e^{n \cdot p(z)} z^v \, dz.
\end{equation}
But we have
\begin{equation} \label{fol}
  \frac{e^{n w}}{(n w)^{n+v}} \G(n+v+1) = \frac{e^{n w}}{(n w)^{n+v}} \int_0^\infty e^{-t} t^{n+v}\, dt
  =
  n w  \int_{\nd} e^{nw(1-z)} z^{n+v} \, dz,
\end{equation}
with $\nd = \nd_w$  the line from $0$ through $1/w$ to infinity. Let $\beta :=\arg(w)$. Then $|\beta|<\pi/2$ and we also have $\arg(1/w)=-\beta$. In the usual way,  the line of integration $\nd$ may be moved to the positive real axis after checking some growth estimates as follows. Let $\nd_1$ be the path from $R_1>0$ to $R_2>R_1$. Then $\nd_2$ is the arc of radius $R_2$ from $R_2$ to $R_2 e^{-i \beta}$. Next $\nd_3$ is the line from $R_2 e^{-i \beta}$ to $R_1 e^{-i \beta}$, coinciding with part of $\nd$, and lastly $\nd_4$ is the arc of radius $R_1$ from $R_1 e^{-i \beta}$ to $R_1$. Integrating $e^{nw(1-z)} z^{n+v}$ around the closed path made up of $\nd_1$, $\nd_2$, $\nd_3$ and $\nd_4$ gives zero since it is holomorphic on the interior. Writing $w=|w|e^{i \beta}$ and $z=R e^{i \theta}$, we may bound the integrals over the arcs $\nd_2$ and $\nd_4$ with
\begin{align*}
  \left| e^{nw(1-z)} z^{n+v}\right| & \lqs e^{n|w|(\cos(\beta)-R \cos(\beta+\theta))}R^{n+\Re(v)} \\
  & \lqs e^{n|w|\cos(\beta)(1-R)}R^{n+\Re(v)}
\end{align*}
since $\theta$ is between $-\beta$ and $0$. Therefore, for $R=R_2$,
\begin{equation} \label{ct3}
  n w  \int_{\nd_2} e^{nw(1-z)} z^{n+v} \, dz \ll n e^{n|w|\cos(\beta)(1-R_2)}R_2^{n+\Re(v)+1}.
\end{equation}
As $\cos(\beta)>0$ we see that \e{ct3} goes to zero as $R_2 \to \infty$. Also the integral over $\nd_4$ goes to zero as $R=R_1 \to 0$ in \e{ct3} when $n$ is large enough that $n+\Re(v)+1>0$. Hence \e{fol} implies \e{ct2} as we wanted.

Lastly we argue as in \e{dan} to bound the part of the integral \e{ct} with $z\gqs 3/2$. Use that $e^{1-z}z \lqs e^{-z/20}$ when $z\gqs 3/2$ to show
\begin{equation*}
  n w  \int_{3/2}^\infty e^{n \cdot p(z)} z^v \, dz \ll  n   \int_{3/2}^\infty  e^{-n z/20} z^{\Re(v)} \, dz.
\end{equation*}
For $n$ large enough that $e^{-n z/40} z^{\Re(v)} \lqs 1$ when $3/2 \lqs z$, we may replace the last integrand by $ e^{-n z/40}$ and complete the proof of \e{lmf2}.
\end{proof}

\section{The case $w=1$} \label{w=1}

In this section we set $w=1$ so that $p(z)=p(1;z)=1-z+\log z$.

\begin{prop} \label{pil}
Let $S$ be a fixed positive integer. As $n \to \infty$
\begin{align}
  \int_{1/2}^1 e^{n \cdot p(z)} z^v \, dz & = \sum_{s=0}^{S-1} \G\left( \frac{s+1}2\right) \frac{(-1)^s\beta_s(v)}{n^{(s+1)/2}} + O\left(\frac 1{n^{(S+1)/2}} \right),\label{in}\\
  \int_1^{3/2} e^{n \cdot p(z)} z^v \, dz & = \sum_{s=0}^{S-1} \G\left( \frac{s+1}2\right) \frac{\beta_s(v)}{n^{(s+1)/2}} + O\left(\frac 1{n^{(S+1)/2}} \right), \label{in2}
\end{align}
for
\begin{equation}
  \beta_s(v) =\sum_{m=0}^s (-1)^m \binom{v}{s-m} \sum_{k=0}^m 2^{(s-1)/2+k} \binom{-(s+1)/2}{k}
  \dm_{m,k}\left( \frac{1}{3}, \frac{1}{4}, \frac{1}{5}, \dots \right). \label{in3}
\end{equation}
\end{prop}
\begin{proof}
We see that $p(z)$ is holomorphic away from $z=0$. Let $z_0=1$ with $p(1)=0$ and we have the Taylor expansion about $1$:
\begin{equation*}
  p(z)-p(1)=-\sum_{j=2}^\infty \frac{(-1)^{j}}{j} (z-1)^{j}.
\end{equation*}
According to our setup with Assumptions \ref{ma0} and \e{pqe},  $q(z)=z^v$ and
\begin{equation}\label{mu2}
p_s=(-1)^s/(s+2),  \quad p_0=1/2, \quad \mu=2, \quad \theta_\ell=\pi \ell,  \quad q_s=\binom{v}{s}.
\end{equation}
 For the integral in \e{in} we may apply Theorem \ref{il} with $\cc$ the interval from $1$ to $1/2$ since $\Re(p(z))$ has its maximum at $z=1$. The initial part of $\cc$ lies in the sector with bisecting angle $\theta_1=\pi$, since the contour is moving left, and we need $k=1$ in \e{wim}.
This means that
\begin{equation*}
     \int_{1}^{1/2} e^{n \cdot p(z)} \, dz = \sum_{s=0}^{S-1}  \G\left(\frac{s+1}{2}\right) \frac{\alpha_s \cdot (-1)^{(s+1)}}{n^{(s+1)/2}} + O\left(\frac{1}{n^{(S+1)/2}} \right)
\end{equation*}
and $\beta_s(v)=\alpha_s$ is computed with Proposition \ref{wojf} to get \e{in3}.

The integral \e{in2} is handled the same way, the only difference being that the contour is moving right into the sector with  bisecting angle $\theta_0=0$,  and so  $k=0$ is needed in \e{wim}.
\end{proof}

With \e{ori}, \e{ori2} and Lemma \ref{snwx}, we may
extend the definition of $\theta_n(v)$ to all $n>0$ and $v\in \C$ with $\Re(n+v)>-1$ using
\begin{align}\label{lie}
   \theta_n(v) & =S_n(1;v) - \frac{e^n}{2 n^{n+v}}\G(n+v+1) \\
   & = \frac{S_n(1;v)}2 -\frac{T_n(1;v)}2. \label{bx}
\end{align}
Define
\begin{equation}\label{rhrv}
  \rho_r(v) := \delta_{r,0}-  \sum_{m=0}^{2r+1} (-1)^m\binom{v}{2r+1-m}\sum_{k=0}^{m}  \frac{(2r+2k)!!}{(-1)^k k!} \dm_{m,k}\left( \frac{1}{3}, \frac{1}{4}, \frac{1}{5}, \dots \right).
\end{equation}

\begin{theorem} \label{ext}
With the above definition, as  $n \to \infty$,
\begin{equation}\label{abb}
  \theta_n(v) =\rho_0(v)+\frac{\rho_1(v)}{n}+\frac{\rho_2(v)}{n^2}+ \cdots + \frac{\rho_{R-1}(v)}{n^{R-1}}+
  O\left( \frac{1}{n^R}\right)
\end{equation}
for an implied constant depending only on $R\in \Z_{\gqs 1}$ and $v\in \C$.
\end{theorem}
\begin{proof}
Together, Lemma \ref{xm},  Proposition \ref{pil}  and \e{bx} imply that
\begin{equation*}
  \theta_n(v) = 1+\frac n2 \left(\sum_{s=0}^{S-1} \G\left( \frac{s+1}2\right) \frac{\beta_s(v)}{n^{(s+1)/2}}((-1)^s -1) + O\left(\frac 1{n^{(S+1)/2}} \right) \right).
\end{equation*}
Summing over $s=2r+1$ odd then gives
\begin{equation*}
  \theta_n(v) = 1-\sum_{r=0}^{R-1} \G(r+1) \frac{\beta_{2r+1}(v)}{n^{r}} + O\left(\frac 1{n^{R}}\right)
\end{equation*}
and the theorem follows.
\end{proof}


\begin{proof}[Proof of Proposition \ref{xpg}]
By \e{snw3},
\begin{equation}
  \G(n+v+1)  = \frac{n^{n+v}}{e^n} \left( S_n(1;v) + T_n(1;v)\right). \label{bx2}
\end{equation}
Hence, for $\beta_s(v)$ in \e{in3},
\begin{equation*}
  \G(n+v+1) =n \frac{n^{n+v}}{e^n} \left(\sum_{s=0}^{S-1} \G\left( \frac{s+1}2\right) \frac{\beta_s(v)}{n^{(s+1)/2}}((-1)^s +1) + O\left(\frac 1{n^{(S+1)/2}} \right) \right).
\end{equation*}
Summing over $s=2r$ even then gives
\begin{equation*}
  \G(n+v+1) = 2\sqrt{n} \frac{n^{n+v}}{e^n} \left(\sum_{r=0}^{R-1} \G(r+1/2) \frac{\beta_{2r}(v)}{n^{r}} + O\left(\frac 1{n^{R}}\right)\right).
\end{equation*}
The proof is completed using
\begin{equation*}
  \G(r+1/2) = \frac{\sqrt{\pi} (2r)!}{2^{2r} r!}, \qquad \binom{-r-1/2}{k}
  =\frac{(-1)^k}{2^{2k} k!} \frac{(2r+2k)! r!}{(2r)! (r+k)!},
\end{equation*}
which, with \e{doub}, imply the identity
\begin{equation}\label{idp}
  2^{r+k}\frac{\G(r+1/2)}{\sqrt{\pi}}\binom{-r-1/2}{k}  =  \frac{(2r+2k-1)!! }{(-1)^k k!}.
\end{equation}
\end{proof}

\begin{cor}[Stirling's approximation] \label{stirgam}
As real $n \to \infty$,
\begin{equation} \label{gmx}
  \G(n+1)= \sqrt{2\pi n}\left(\frac{n}{e}\right)^n \left(1+\frac{\g_1}{n}+\frac{\g_2}{n^2}+ \cdots + \frac{\g_{R-1}}{n^{R-1}} +O\left(\frac{1}{n^{R}}\right)\right),
\end{equation}
for an implied constant depending only on $R$, with
\begin{equation}\label{gmy}
  \g_r = \sum_{k=0}^{2r}  \frac{(2r+2k-1)!!}{(-1)^k k!} \dm_{2r,k}\left( \frac{1}{3}, \frac{1}{4}, \frac{1}{5}, \dots \right).
\end{equation}
\end{cor}

Corollary \ref{stirgam} is the $v=0$ case of Proposition \ref{xpg}, and an equivalent form of \e{gmy} is due to Perron \cite[p. 210]{Pe17}. Also \e{gmy} is equivalent to \cite[Thm. 2.7]{BM11} using the generating function $z+\log(1-z)$. Brassesco and M\'endez give another formulation in \cite[Thm. 2.1]{BM11}, based on $e^z-1-z$, and the formula corresponding to \e{gmy} is the same, except that $3$, $4$, $5, \dots$ are replaced by factorials:
\begin{equation}\label{gmy2}
  \g_r = \sum_{k=0}^{2r}  \frac{(2r+2k-1)!!}{(-1)^k k!} \dm_{2r,k}\left( \frac{1}{3!}, \frac{1}{4!}, \frac{1}{5!}, \dots \right).
\end{equation}
This is true even though $\dm_{2r,k}\left( \frac{1}{3}, \frac{1}{4}, \frac{1}{5}, \dots \right)$ and $\dm_{2r,k}\left( \frac{1}{3!}, \frac{1}{4!}, \frac{1}{5!}, \dots \right)$ are  usually not equal. We will see \e{gmy2} as the  $v=0$ case of Proposition \ref{v0}.

In this section, with $w=1$, the function $p(1;z)$ has a simple saddle-point at $z=1$, i.e. $\frac d{dz} p(1;z)|_{z=1}=0$. This means $\mu=2$ in \e{mu2}. In the next section, where $w\neq 1$, the function $p(w;z)$ will no longer have a saddle-point at $z=1$, only a maximum. This makes $\mu=1$ and changes the shape of the asymptotics as we will see. Soni and Soni  show  in \cite{So92} how to give an asymptotic expansion for $S_n(w;0)$ that is uniform for $w$ in a neighborhood of $1$.

\section{The cases $w \neq 1$}

Define the rational functions
\begin{equation}\label{wb}
  U_r(w;v) := \delta_{r,0}-\sum_{m=0}^r (-1)^{m}\binom{v}{r-m}\sum_{k=0}^{m} \frac{ w}{(w-1)^{r+k+1}}
   \frac{(r+k)!}{ k!}  \dm_{m,k}\left( \frac{1}{2}, \frac{1}{3}, \frac{1}{4},  \dots \right),
\end{equation}
so that, for example,
\begin{gather*}
  U_0(w;v)=\frac{1}{1-w}, \qquad U_1(w;v)=-\frac{w}{(1-w)^3} - v\frac{w}{(1-w)^2},\\
  U_2(w;v)=\frac{w(2w+1)}{(1-w)^5} + v \frac{w(w+2)}{(1-w)^4}+ v^2 \frac{w}{(1-w)^3}.
\end{gather*}

\begin{prop} \label{abc}
 Let $w$ and $v$ be fixed complex numbers with $\Re(w)\lqs 1$ and $w \neq 1$.  Then as real $n \to \infty$,
\begin{equation} \label{wa}
  S_n(w;v) = U_0(w;v) + \frac{U_1(w;v)}{n} + \frac{U_2(w;v)}{n^2} + \cdots + \frac{U_{R-1}(w;v)}{n^{R-1}}  +O\left(\frac{1}{n^{R}}\right),
\end{equation}
for an implied constant depending only on $w$, $v$ and $R$.
\end{prop}
\begin{proof}
Recall that $p(z)= w(1-z)+\log z$.
Theorem \ref{il} may be applied  to \e{lmf} since $\Re(p(z))$ is strictly increasing for $1/2\lqs z \lqs 1$ when $\Re(w)\lqs 1$.
As $p(1)=0$, the expansion of $p(z)$ at $z=1$ can be written as
\begin{equation*}
  p(z)-p(1)=-(w-1)(z-1)-\sum_{j=2}^\infty \frac{(-1)^{j}}{j} (z-1)^{j}.
\end{equation*}
This time our setup with Assumptions \ref{ma0} and \e{pqe} has
\begin{equation}\label{mu1}
p_0=w-1, \quad \mu=1,  \quad p_s=(-1)^{s+1}/(s+1) \ \text{ for } \ s\gqs 1,  \quad \theta_\ell=-\arg(w-1)+2\pi \ell.
\end{equation}
Also $q_s=\binom{v}{s}$, $\cc$ is the interval from $1$ to $1/2$  and $k=0$ in \e{wim}; when $\mu=1$ then $k$ can be any integer and there is  only `one' direction with $\Re(p(z))$ decreasing.
Computing $\alpha_s$ with Proposition \ref{wojf} and simplifying with \e{mulk} and \e{muln} completes the proof.
\end{proof}

\begin{prop} \label{abc2}
 Let $w$ and $v$ be fixed complex numbers with $\Re(w)\gqs 1$.  Then as real $n \to \infty$,
\begin{equation} \label{wat}
  T_n(w;v) = -U_0(w;v) - \frac{U_1(w;v)}{n} - \frac{U_2(w;v)}{n^2} - \cdots - \frac{U_{R-1}(w;v)}{n^{R-1}}  +O\left(\frac{1}{n^{R}}\right),
\end{equation}
for an implied constant depending only on $w$, $v$ and $R$.
\end{prop}
\begin{proof}
Theorem \ref{il} may be applied  to \e{lmf2} since $\Re(p(z))$ is strictly decreasing for $1\lqs z \lqs 3/2$ when $\Re(w)\gqs 1$. The same calculation  as for Proposition \ref{abc} gives the result.
\end{proof}

Our asymptotics for $S_n(w;v)$ and $T_n(w;v)$ can now be assembled. The functions $\rho_r(v)$, $\g_r(v)$ and $U_r(w;v)$ are defined in \e{rhrv}, \e{gnv2} and \e{wb} respectively. Also recall the partition of the $w$-plane shown in Figure \ref{fig}.
If $|w e^{1-w}|=1$, write $w e^{1-w}=e^{-i \varphi(w)}$ with $\varphi(w)$ real.

\begin{theorem} \label{msth}
Let $w$ and $v$ be complex numbers. As real $n \to \infty$,
\begin{align}
  S_n(1;v) & =\sum_{r=0}^{R-1} \frac{1}{n^r} \left( \rho_r(v) + \frac{\g_r(v)}{2}\sqrt{2\pi n} \right) + O\left( \frac{1}{n^{R-1/2}}\right), \label{s1}\\
  S_n(w;v) & = \sum_{r=0}^{R-1} \frac{U_r(w;v)}{n^r}   + O\left( \frac{1}{n^R}\right) \qquad (w\in \mathcal X \cup \mathcal Y \cup \mathcal S), \label{s2}\\
  S_n(w;v) & = \sum_{r=0}^{R-1} \frac{1}{n^r} \left( U_r(w;v) + \g_r(v)  \frac{e^{n i  \cdot \varphi(w)}}{w^v} \sqrt{2\pi n}  \right) + O\left( \frac{1}{n^{R-1/2}}\right)\qquad (w\in \mathcal T), \label{s3}\\
  S_n(w;v) & = \frac{\sqrt{2\pi n} }{(w e^{1-w})^n w^v}
  \left(\sum_{r=0}^{R-1} \frac{\g_r(v)}{ n^r}  + O\left(\frac 1{n^{R}} \right)\right)\qquad (w\in \mathcal Z). \label{s4}
\end{align}
\end{theorem}
\begin{proof}
The asymptotic \e{s1} is a consequence of \e{lie}, Theorem \ref{ext} and Proposition \ref{xpg}. It can be seen from Proposition \ref{abc} that \e{s2} is true for $\Re(w)\lqs 1$ and $w \neq 1$ and this includes $\mathcal Y$, $\mathcal S$ and  part of $\mathcal X$. For $\Re(w)> 1$, starting from \e{snw3} and using Propositions \ref{xpg}, \ref{abc2},
\begin{align}
  S_n(w;v) & =\frac{e^{n w}}{(n w)^{n+v}} \G(n+v+1) - T_n(w;v) \notag\\
  & =\frac{e^{n w}}{(n w)^{n+v}}\sqrt{2\pi n}\frac{n^{n+v}}{e^n} \left(\sum_{r=0}^{R-1}\frac{\g_r(v)}{n^r}+ O\left(\frac{1}{n^{R}}\right)\right) + \sum_{r=0}^{R-1}\frac{U_r(w;v)}{n^r}+O\left(\frac{1}{n^{R}}\right) \notag\\
  & =\frac{\sqrt{2\pi n}}{(w e^{1-w})^{n}w^v} \left(\sum_{r=0}^{R-1}\frac{\g_r(v)}{n^r}+ O\left(\frac{1}{n^{R}}\right)\right) + \sum_{r=0}^{R-1}\frac{U_r(w;v)}{n^r}+O\left(\frac{1}{n^{R}}\right).\label{rrr}
\end{align}
If $|w e^{1-w}|>1$ then the part of \e{rrr} containing this factor decays exponentially and we obtain \e{s2} for the remaining piece of $\mathcal X$. The case $|w e^{1-w}|=1$ in \e{rrr} gives \e{s3} and lastly $|w e^{1-w}|<1$ in \e{rrr} gives \e{s4}.
\end{proof}

\begin{theorem} \label{mtth}
Let $w$ and $v$ be complex numbers. As real $n \to \infty$,
\begin{align}
  T_n(1;v) & =\sum_{r=0}^{R-1} \frac{1}{n^r} \left( -\rho_r(v) + \frac{\g_r(v)}{2}\sqrt{2\pi n} \right) + O\left( \frac{1}{n^{R-1/2}}\right), \label{t1} \\
  T_n(w;v) & = -\sum_{r=0}^{R-1} \frac{U_r(w;v)}{n^r}   + O\left( \frac{1}{n^R}\right) \qquad (w\in \mathcal X \cup \mathcal Z \cup \mathcal T), \label{t2} \\
  T_n(w;v) & = \sum_{r=0}^{R-1} \frac{1}{n^r} \left( -U_r(w;v) + \g_r(v)  \frac{e^{n i \cdot \varphi(w)}}{w^v} \sqrt{2\pi n}  \right) + O\left( \frac{1}{n^{R-1/2}}\right)\qquad (w\in \mathcal S), \label{t3} \\
  T_n(w;v) & = \frac{\sqrt{2\pi n} }{(w e^{1-w})^n w^v}
  \left(\sum_{r=0}^{R-1} \frac{\g_r(v)}{ n^r}  + O\left(\frac 1{n^{R}} \right)\right)\qquad (w\in \mathcal Y, \ w\neq 0). \label{t4}
\end{align}
\end{theorem}
\begin{proof}
The asymptotic \e{t1} is a consequence of \e{snw3},  Proposition \ref{xpg} and \e{s1}.  It can be seen from Proposition \ref{abc2} that \e{t2} is true for $\Re(w)\gqs 1$ and $w \neq 1$ and this includes $\mathcal Z$, $\mathcal T$ and  part of $\mathcal X$. For $\Re(w)< 1$ and $w\neq 0$, starting from \e{snw3} and using Propositions \ref{xpg}, \ref{abc},
\begin{align}
  T_n(w;v) & =\frac{e^{n w}}{(n w)^{n+v}} \G(n+v+1) - S_n(w;v) \notag\\
  & =\frac{\sqrt{2\pi n}}{(w e^{1-w})^{n}w^v} \left(\sum_{r=0}^{R-1}\frac{\g_r(v)}{n^r}+ O\left(\frac{1}{n^{R}}\right)\right) - \sum_{r=0}^{R-1}\frac{U_r(w;v)}{n^r}+O\left(\frac{1}{n^{R}}\right).\label{rrr2}
\end{align}
The remaining cases of \e{t2}, \e{t3} and \e{t4} follow from \e{rrr2} depending on the size of $|w e^{1-w}|$.
\end{proof}

Theorems \ref{msth} and \ref{mtth} agree with \cite[Eqs. (5), (6)]{Bu63} for $w\neq 1$ and $v=0$. His function $U_r(z)$ equals $U_r(z,0)-\delta_{r,0}$ here.
Since $\g_0(v)=1$ we may state the following corollaries.
\begin{cor}
For all $w$, $v\in \C$ we have that $S_n(w;v)$ remains bounded as $n\to \infty$ if and only if $w\in \mathcal X \cup \mathcal Y \cup \mathcal S$. Also $T_n(w;v)$ remains bounded  if and only if $w\in \mathcal X \cup \mathcal Z\cup \mathcal T$.
\end{cor}

\begin{cor}
For all $w$, $v\in \C$ we have that $S_n(w;v)/\sqrt{n}$ remains bounded as $n\to \infty$ if and only if $w\notin \mathcal Z$. Also $T_n(w;v)/\sqrt{n}$ remains bounded  if and only if $w\notin \mathcal Y$.
\end{cor}


\section{Another description of $\rho_r(v)$, $\g_r(v)$ and $U_r(w;v)$} \label{six}
As we have seen in Theorems \ref{msth} and \ref{mtth}, the asymptotics of $S_n(w;v)$ and $T_n(w;v)$ can be completely described in terms of the functions
\begin{align*}
  \rho_r(v)= \delta_{r,0}-{} & \sum_{m=0}^{2r+1} (-1)^m\binom{v}{2r+1-m}\sum_{k=0}^{m} \frac{(2r+2k)!!}{(-1)^k k!} \dm_{m,k}\left( \frac{1}{3}, \frac{1}{4}, \frac{1}{5}, \dots \right), \\
  \g_r(v)  =  & \sum_{m=0}^{2r} (-1)^m \binom{v}{2r-m}\sum_{k=0}^{m}  \frac{(2r+2k-1)!!}{(-1)^k k!} \dm_{m,k}\left( \frac{1}{3}, \frac{1}{4}, \frac{1}{5}, \dots \right),\\
  U_r(w;v)  = \delta_{r,0}-{} & \sum_{m=0}^r (-1)^{m}\binom{v}{r-m}\sum_{k=0}^{m}  \frac{ w}{(w-1)^{r+k+1}}
  \frac{(r+k)!}{ k!}  \dm_{m,k}\left( \frac{1}{2}, \frac{1}{3}, \frac{1}{4},  \dots \right).
\end{align*}
In this section we find variants of these formulas where the numbers $1/j$ inside the De Moivre polynomials are replaced by $1/j!$.

Integrating \e{snw2} by parts shows
\begin{align}
  S_n(w;v) & = 1+n w  \int_0^1 e^{n w(1-z)} z^{n+v} \, dz \notag\\
   & = (n+v) \int_0^1 e^{n w(1-z)} z^{n+v-1} \, dz  \notag\\
   & = (n+v) \int_{-\infty}^0 e^{n w(1-e^z)} e^{(n+v)z} \, dz, \label{here}
\end{align}
after the change of variables $z\to e^z$.
Define
\begin{equation}\label{pz2}
  \tilde p(z)= \tilde p(w;z) := w(1-e^z)+  z.
\end{equation}

\begin{lemma} \label{par}
Let $w$, $v$ be in $\C$. As real $n \to \infty$,
\begin{alignat}{2} \label{par2}
    S_n(w;v) & = (n+v)   \int_{-2}^0 e^{n \cdot \tilde p(z)} e^{v z} \, dz +O\left(e^{-n/2}\right) \qquad & &(\Re(w)\lqs 1),\\
    T_n(w;v) & = (n+v)   \int_{0}^2 e^{n \cdot \tilde p(z)} e^{v z} \, dz +O\left(e^{-n/2}\right) \qquad & &(\Re(w)\gqs 1), \label{par3}
\end{alignat}
for implied constants depending only on $w$ and $v$.
\end{lemma}
\begin{proof}
The proof is based on \e{snw3} and \e{here}. It is similar to that of Lemma \ref{xm}, requiring the inequality $1+z-e^z < -|z|/2$ for real $z$ with  $|z|\gqs 2$.
\end{proof}

\begin{prop} \label{pilcc}
Let $S \in \Z_{\gqs 1}$ be  fixed  and set $w=1$ so that $\tilde p(z)= \tilde p(1;z) := 1-e^z+  z$. As $n \to \infty$,
\begin{align}
  \int_{-2}^0 e^{n \cdot \tilde p(z)} e^{v z} \, dz & = \sum_{s=0}^{S-1} \G\left( \frac{s+1}2\right) \frac{(-1)^s \tilde \beta_s(v)}{n^{(s+1)/2}} + O\left(\frac 1{n^{(S+1)/2}} \right),\label{incc}\\
   \int_{0}^2 e^{n \cdot \tilde p(z)} e^{v z} \, dz & = \sum_{s=0}^{S-1} \G\left( \frac{s+1}2\right) \frac{ \tilde \beta_s(v)}{n^{(s+1)/2}} + O\left(\frac 1{n^{(S+1)/2}} \right), \label{in2cc}
\end{align}
for
\begin{equation}
   \tilde \beta_s(v) =\sum_{m=0}^s \frac{v^{s-m}}{(s-m)!} \sum_{k=0}^m 2^{(s-1)/2+k} \binom{-(s+1)/2}{k}
  \dm_{m,k}\left( \frac{1}{3!}, \frac{1}{4!}, \frac{1}{5!}, \dots \right). \label{in3cc}
\end{equation}
\end{prop}
\begin{proof}
The function $ \tilde p(z)$ is entire. Let $z_0=0$ with $ \tilde p(0)=0$ and we have the Taylor expansion about $0$:
\begin{equation*}
   \tilde p(z)- \tilde p(0)=-\sum_{j=2}^\infty \frac{z^{j}}{j!}.
\end{equation*}
According to our setup with Assumptions \ref{ma0} and \e{pqe}, $p(z)= \tilde p(z)$, $q(z)=e^{v z}$ and
$$
p_s=1/(s+2)!,  \quad p_0=1/2, \quad \mu=2, \quad \theta_\ell=\pi \ell,  \quad q_s= v^s/s!.
$$
The proof continues by applying Theorem \ref{il} and Proposition \ref{wojf}, similarly to the proof of Proposition \ref{pil}.
\end{proof}

Define
\begin{align} \label{rhk}
   \tilde \rho_r(v) & := -\sum_{m=0}^{2r+1} \frac{v^{2r+1-m}}{(2r+1-m)!}\sum_{k=0}^{m}  \frac{(2r+2k)!!}{(-1)^k k!} \dm_{m,k}\left( \frac{1}{3!}, \frac{1}{4!}, \frac{1}{5!}, \dots \right), \\
  \tilde \g_r(v) & := \sum_{m=0}^{2r} \frac{v^{2r-m}}{(2r-m)!}\sum_{k=0}^{m}  \frac{(2r+2k-1)!!}{(-1)^k k!} \dm_{m,k}\left( \frac{1}{3!}, \frac{1}{4!}, \frac{1}{5!}, \dots \right). \label{gmk}
\end{align}

\begin{prop} \label{v0}
The asymptotic expansion coefficients $\rho_r(v)$  of $\theta_n(v)$ in \e{abb},  and $\g_r(v)$ of $\G(n+v+1)$ in \e{gnv},  satisfy $\rho_0(v)=\tilde\rho_0(v)$, $\g_0(v) = \tilde \g_0(v)$ and
\begin{equation*}
  \rho_r(v) = \tilde \rho_r(v) + v \cdot \tilde \rho_{r-1}(v), \qquad
  \g_r(v) = \tilde \g_r(v) + v \cdot \tilde \g_{r-1}(v) \qquad (r\gqs 1).
\end{equation*}
\end{prop}
\begin{proof}
Start with  \e{bx} and use Lemma \ref{par}, Proposition \ref{pilcc} to show that
\begin{align*}
  \theta_n(v) & = \frac{n+v}2 \left(\sum_{s=0}^{S-1} \G\left( \frac{s+1}2\right) \frac{ \tilde\beta_s(v)}{n^{(s+1)/2}}((-1)^s -1) + O\left(\frac 1{n^{(S+1)/2}} \right)\right)\\
   & =  -\left(1+\frac{v}n\right) \left(\sum_{r=0}^{R-1} \G\left( r+1\right) \frac{ \tilde\beta_{2r+1}(v)}{n^{r}} + O\left(\frac 1{n^{R}} \right)\right).
\end{align*}
A calculation finds $-\G\left( r+1\right) \tilde\beta_{2r+1}(v) = \tilde \rho_r(v)$ and we obtain the desired relations for $\tilde \rho_r(v)$.
The results for $\tilde \g_r(v)$ are shown similarly, starting with \e{bx2}.
\end{proof}

Set
\begin{equation} \label{urt}
   \tilde U_r(w;v)  := -\sum_{m=0}^{r} \frac{v^{r-m}}{(r-m)!}\sum_{k=0}^{m}   \frac{(-w)^k}{(w-1)^{r+k+1}}
   \frac{(r+k)!}{ k!}
  \dm_{m,k}\left( \frac{1}{2!}, \frac{1}{3!}, \frac{1}{4!}, \dots \right).
\end{equation}

\begin{prop} \label{urt2}
We have $U_0(w;v) = \tilde U_0(w;v)$ and
\begin{equation*}
  U_r(w;v) = \tilde U_r(w;v) + v \cdot \tilde U_{r-1}(w;v) \qquad (r\gqs 1).
\end{equation*}
\end{prop}
\begin{proof}
We will use \e{par2} to give an asymptotic expansion of $S_n(w;v)$ for $\Re(w)\lqs 1$ and $w\neq 1$. The function $\tilde p(z)=w(1-e^z)+  z$ is entire. Let $z_0=0$ with $ \tilde p(0)=0$ and we have the Taylor expansion:
\begin{equation*}
   \tilde p(z)- \tilde p(0)=-(w-1)z - w\sum_{j=2}^\infty \frac{z^{j}}{j!}.
\end{equation*}
In the notation of Assumptions \ref{ma0} and \e{pqe}, $p(z)= \tilde p(z)$, $q(z)=e^{v z}$ and
$$
p_0=w-1, \quad \mu=1,  \quad p_s=w/(s+1)! \ \text{ for } \ s\gqs 1,  \quad q_s= v^s/s!.
$$
The proof continues by applying Theorem \ref{il} and Proposition \ref{wojf}, and comparing the resulting series with Proposition \ref{abc}.
\end{proof}

\section{Further formulas when $v=0$}
In this section we set $v=0$ and then omit $v$ from the notation. The formulas \e{rhrv}, and \e{rhk} with Proposition \ref{v0}, simplify to give Ramanujan's coefficients in \e{rb2} as
\begin{align} \label{rfb}
   \rho_r   = \delta_{r,0} +{} & \sum_{k=0}^{2r+1}  \frac{(2r+2k)!!}{(-1)^k k!} \dm_{2r+1,k}\left( \frac{1}{3}, \frac{1}{4}, \frac{1}{5}, \dots \right) \\
    = - & \sum_{k=0}^{2r+1}  \frac{(2r+2k)!!}{(-1)^k k!} \dm_{2r+1,k}\left( \frac{1}{3!}, \frac{1}{4!}, \frac{1}{5!}, \dots \right). \label{rfb2}
\end{align}
The similar expressions for $\g_r$ in Stirling's approximation \e{gmx} have already been noted in \e{gmy}, \e{gmy2}.
Also with \e{wb}, \e{urt} and Proposition \ref{urt2}, Buckholtz's functions from \cite{Bu63} can be written as
\begin{align}
   U_r(w)   = \delta_{r,0} +{} &\sum_{k=0}^{r} \frac{ w}{(1-w)^{r+k+1}}  \cdot \frac{(r+k)!}{(-1)^{k} k!}  \cdot  \dm_{r,k}\left( \frac{1}{2}, \frac{1}{3}, \frac{1}{4},  \dots \right) \label{uj}\\
   =  & \sum_{k=0}^{r} \frac{w^k}{(1-w)^{r+k+1}} \cdot \frac{(r+k)!}{(-1)^{r} k!}
    \cdot \dm_{r,k}\left( \frac{1}{2!}, \frac{1}{3!}, \frac{1}{4!}, \dots \right). \label{uj2}
\end{align}

The De Moivre polynomials above can be expressed in terms of the more familiar Stirling numbers. Recall that the Stirling subset numbers $\stirb{n}{k}$  count the number of ways to partition  $n$ elements into $k$ nonempty subsets. The Stirling cycle numbers $\stira{n}{k}$ count the number of ways to arrange $n$ elements into $k$ cycles. Their properties are developed in \cite[Sect. 6.1]{Knu}. Also described in \cite[Sect. 6.2]{Knu} and \cite{GS78} are the second-order Eulerian numbers $\eud{n}{k}$ and  the relations
\begin{align}\label{eu2a}
  \stira{r+j}{j} & = \sum_{k=0}^r \eu{r}{k} \binom{r+j+k}{2r}, \\
  \stirb{r+j}{j} & = \sum_{k=0}^r \eu{r}{k} \binom{2r+j-1-k}{2r}. \label{eu2b}
\end{align}
As shown in \cite[Sect. 2]{odm}, for example,
\begin{equation} \label{ank}
  \dm_{n,k}\left( \frac{1}{1}, \frac{1}{2}, \frac{1}{3},  \dots \right)  =
  \frac{k!}{n!}\stira{n}{k}, \qquad
  \dm_{n,k}\left( \frac{1}{1!}, \frac{1}{2!}, \frac{1}{3!},  \dots \right)  = \frac{k!}{n!}\stirb{n}{k}.
\end{equation}
Removing the first coefficient $a_1$ in $\dm_{n,k}(a_1,a_2, \dots)$ is easily achieved by the binomial theorem:
\begin{equation}\label{add}
  \dm_{n,k}(a_2,a_3, \dots)  = \sum_{j=0}^k (-a_1)^{k-j} \binom{k}{j} \dm_{n+j,j}(a_1,a_2, \dots).
\end{equation}
Then it follows from \e{ank} and \e{add} that
\begin{align}
\dm_{r,k}\left(  \frac{1}{2}, \frac{1}{3}, \frac{1}{4}, \dots \right)
  & =\frac{k!}{(r+k)!} \sum_{j=0}^k (-1)^{k-j}  \binom{r+k}{r+j} \stira{r+j}{j},\label{jfa}\\
  \dm_{r,k}\left(  \frac{1}{2!}, \frac{1}{3!}, \frac{1}{4!}, \dots \right)
  & =\frac{k!}{(r+k)!} \sum_{j=0}^k (-1)^{k-j}  \binom{r+k}{r+j} \stirb{r+j}{j}.\label{jfb}
\end{align}

\begin{prop} \label{rop}
For $w\in \C$,
\begin{alignat}{2}
  U_r(w) & = \delta_{r,0} +\frac{(-1)^r w}{(1-w)^{2r+1}} \sum_{j=0}^r \eu{r}{j} w^{j} \qquad & &(w\neq 1), \label{kn}\\
   & = \delta_{r,0} +(-1)^r\sum_{j=1}^\infty \stirb{r+j}{j} w^j\qquad & &(|w|<1). \label{car}
\end{alignat}
\end{prop}
\begin{proof}
Inserting \e{eu2a} and \e{eu2b} into \e{jfa} and  \e{jfb}, respectively, and simplifying the binomial sum as in \cite[Eq. (5.24)]{Knu} reveals that
\begin{align}
\dm_{r,k}\left(  \frac{1}{2}, \frac{1}{3}, \frac{1}{4}, \dots \right)
  & =\frac{k!}{(r+k)!} \sum_{j=0}^r   \eu{r}{j} \binom{j}{r-k},\label{yta}\\
  \dm_{r,k}\left(  \frac{1}{2!}, \frac{1}{3!}, \frac{1}{4!}, \dots \right)
  & =\frac{k!}{(r+k)!} \sum_{j=0}^r   \eu{r}{j} \binom{r-1-j}{r-k}.\label{ytb}
\end{align}
Then substituting \e{yta} into \e{uj}, (or \e{ytb} into \e{uj2}), interchanging summations and simplifying yields \e{kn}. Expanding $(1-w)^{-2r-1}$ in \e{kn} with the binomial theorem and using \e{eu2b} then shows \e{car}.
\end{proof}

The identity \e{kn} is due to Knuth  and described in \cite[p. 506]{Knuprog}  where the functions $$Q_w(n):=T_n(1/w;0)/w, \qquad
R_w(n):=S_n(w;0)
$$
are studied. Also \e{car} is due to Carlitz in \cite{Ca65}. Proposition \ref{rop} gives new proofs of these  identities. 

Comparing coefficients of $w$ in \e{uj}, \e{uj2} and \e{kn} finds
\begin{alignat}{2}\label{uk}
  \eu{r}{j} & = \sum_{k=0}^{r} (-1)^{r+j+k} \frac{(r+k)!}{k!}\binom{r-k}{j} \dm_{r,k}\left( \frac{1}{2}, \frac{1}{3}, \frac{1}{4},  \dots \right) \qquad &(r &\gqs 0),\\
  & = \sum_{k=0}^{r} (-1)^{j+k+1} \frac{(r+k)!}{k!}\binom{r-k}{j+1-k} \dm_{r,k}\left( \frac{1}{2!}, \frac{1}{3!}, \frac{1}{4!},  \dots \right) \qquad &(r & \gqs 1). \label{uk2}
\end{alignat}
Combining \e{add} with \e{jfa}, \e{jfb}, \e{yta} or \e{ytb} also gives explicit formulas for
$\dm_{r,k}\left( \frac{1}{3}, \frac{1}{4}, \frac{1}{5}, \dots \right)$ and $\dm_{r,k}\left( \frac{1}{3!}, \frac{1}{4!}, \frac{1}{5!}, \dots \right)$. For example,
\begin{equation}\label{wew}
  \dm_{r,k}\left( \frac{1}{3}, \frac{1}{4}, \frac{1}{5}, \dots \right)
  = \sum_{j_1+j_2+j_3=k} \frac{(-1)^{j_1+j_2}}{2^{j_1}} \frac{k!}{j_1! j_2! (r+j_2+2j_3)!} \stira{r+j_2+2j_3}{j_3},
\end{equation}
where the sum is over all $j_1$, $j_2$, $j_3 \in \Z_{\gqs 0}$ that sum to $k$.

Applying \e{add} $r$ times leads to the following result.

\begin{prop}
For $r, k \gqs 0$,
$\dm_{n,k}(a_{r+1},a_{r+2}, \dots)$ equals
\begin{equation}\label{fortx}
   \sum_{ j_1+ j_2+ \dots +j_{r+1}= k}
 \binom{k}{j_1 , j_2 ,  \dots , j_{r+1}} (-a_1)^{j_1} (-a_2)^{j_2}  \cdots (-a_r)^{j_r}
 \dm_{n+J+r j_{r+1},j_{r+1}}(a_{1},a_{2}, \dots),
\end{equation}
where $J$ means $(r-1)j_1+(r-2)j_2+ \cdots +1 j_{r-1}$ and the summation is   over all  $j_1$,  \dots , $j_{r+1} \in \Z_{\gqs 0}$ with sum $k$.
\end{prop}

Using this, along with the equality
$
\dm_{m+k,k}\left( \frac{1}{0!}, \frac{1}{1!}, \frac{1}{2!}, \dots \right) =  k^m/m!
$,
proves the additional formulas
\begin{align}
  \dm_{n,k}\left( \frac{1}{2!}, \frac{1}{3!}, \frac{1}{4!}, \dots \right) \label{mvw}
& =\sum_{j_1+j_2+j_3=k} \binom{k}{j_1,j_2,j_3} (-1)^{j_1+j_2} \frac{j_3^{n+j_1+j_3}}{(n+j_1+j_3)!}, \\
  \dm_{n,k}\left( \frac{1}{3!}, \frac{1}{4!}, \frac{1}{5!}, \dots \right)
& =\sum_{j_1+j_2+j_3+j_4=k} \binom{k}{j_1,j_2,j_3,j_4} \frac{(-1)^{j_1+j_2+j_3}}{2^{j_3}} \frac{j_4^{n+2j_1+j_2+2j_4}}{(n+2j_1+j_2+2j_4)!}. \label{mvw2}
\end{align}

There is also a nice combinatorial interpretation of the De Moivre polynomial values appearing in this section. Let $n$, $k$ and $r$ be integers with $n$, $k\gqs 0$, $r\gqs 1$. Write $\stirb{n}{k}_{\gqs r}$ for the number of ways to partition  $n$ elements into $k$  subsets, each with at least $r$ members. Let $\stira{n}{k}_{\gqs r}$ denote the number of ways to arrange $n$ elements into $k$ cycles, each of length at least $r$. We also set $\stirb{0}{k}_{\gqs r}=\stirb{0}{k}_{\gqs r}=\delta_{k,0}$. These are the so-called $r$-associated Stirling numbers, generalizing the usual $r=1$ case; see for example \cite[pp. 222, 257]{Comtet} and \cite[Sect. 2]{BM11}.
We are following the Knuth-approved notation of \cite{knu95}.

\begin{prop}
Using sequences starting with $r-1$ zeros, we have
\begin{align}
  \stirb{n}{k}_{\!\gqs r} & =\frac{n!}{k!} \dm_{n,k}\left(0,0, \dots, 0, \frac 1{r!}, \frac 1{(r+1)!}, \dots \right) \label{rso}\\
  & =\frac{n!}{k!} \dm_{n-(r-1)k,k}\left(\frac 1{r!}, \frac 1{(r+1)!}, \dots \right),\label{rso2}\\
  \stira{n}{k}_{\!\gqs r} & =\frac{n!}{k!} \dm_{n,k}\left(0,0, \dots, 0, \frac 1{r}, \frac 1{r+1}, \dots \right) \label{rso3}\\
  & =\frac{n!}{k!} \dm_{n-(r-1)k,k}\left(\frac 1{r}, \frac 1{r+1}, \dots \right). \label{rso4}
\end{align}
\end{prop}
\begin{proof}
Express the right hand side of \e{rso} using \e{bell2}. Each nonzero summand corresponds to a partition of $n$ elements into $j_r$ subsets of size $r$, $j_{r+1}$ subsets of size $r+1$, and so on, with $k$ subsets altogether.
Then
\begin{equation*}
  \frac{n!}{(j_r! j_{r+1}! \cdots )(r!)^{j_r} ((r+1)!)^{j_{r+1}} \cdots}
\end{equation*}
counts the $n!$ ways to put the $n$ elements into this partition, dividing by the $j_m!$ ways to order the subsets of size $m$ and dividing by the $m!$ ways to order the elements of each subset of size $m$. This gives the desired $r$-associated Stirling subset number. The argument for the $r$-associated Stirling cycle number in \e{rso3} is the same except that there are only $m$ ways to write a particular cycle of length $m$. The formulas \e{rso2}, \e{rso4} follow by \e{gsb}.
\end{proof}

The De Moivre polynomial values in \e{gmy} -- \e{gmy2} and \e{rfb} -- \e{uj2} can now be replaced by $r$-associated Stirling numbers. For example,
\begin{equation}\label{ass}
  \g_j  = \sum_{k=0}^{2j} \frac{(-1)^k}{(2j+2k)!!} \stira{2j+2k}{k}_{\!\gqs 3}
  , \qquad
  \g_j  = \sum_{k=0}^{2j} \frac{(-1)^k}{(2j+2k)!!} \stirb{2j+2k}{k}_{\!\gqs 3},
\end{equation}
where the first formula in \e{ass} is due to Comtet \cite[p. 267]{Comtet}, and the second  is due to Brassesco and M\'endez  \cite[Thm. 2.4]{BM11}.
Also
\begin{align}\label{ass2}
   \rho_j = \delta_{j,0}  +{} & \sum_{k=0}^{2j+1} \frac{(-1)^k}{(2j+2k+1)!!} \stira{2j+2k+1}{k}_{\!\gqs 3},\\
     = - & \sum_{k=0}^{2j+1} \frac{(-1)^k}{(2j+2k+1)!!} \stirb{2j+2k+1}{k}_{\!\gqs 3}. \label{ass3}
\end{align}

\section{Approximations to the exponential integral}

Ramanujan's next result after \e{rb} and \e{rb2} is  Entry 49, and it seems to have attracted much less attention than Entry 48. Use the relation
\begin{equation}\label{ei}
  1+\frac{1!}{n}+\frac{2!}{n^2}+ \cdots +\frac{(n-1)!}{n^{n-1}} +  \frac{n!}{n^n}\Psi_n = n e^{-n}\ei(n),
\end{equation}
to define $\Psi_n$, with $\ei(n)$ given in \e{ein}. Then Ramanujan  computed the first terms in the asymptotic expansion of $\Psi_n$, writing\footnote{He was considering $1+\Psi_n$ so his first term is $2/3$.}
\begin{equation}\label{ei2}
   \Psi_n  = -\frac{1}3+ \frac{4}{135 n}+\frac{8}{2835 n^2}+O\left( \frac{1}{n^3}\right).
\end{equation}
See Berndt's  discussion \cite[p. 184]{Ber89} of this entry, and a proof of \e{ei2} based on Olver's work in \cite[pp. 523 -- 531]{Olv}.
We are also interested in the generalization
\begin{equation}\label{ei3}
  1+\frac{1!}{n}+\frac{2!}{n^2}+ \cdots +\frac{(n+v-1)!}{n^{n+v-1}} +  \frac{(n+v)!}{n^{n+v}}\Psi_n(v) = n e^{-n}\ei(n),
\end{equation}
and our  goal is to establish the next result.

\begin{theorem} \label{xte}
Let $v$ be any integer. As real $n \to \infty$,
\begin{equation} \label{trw}
  \Psi_n(v) = \psi_0+\frac{\psi_1(v)}{n}+\frac{\psi_2(v)}{n^2}+ \cdots + \frac{\psi_{R-1}(v)}{n^{R-1}} +O\left(\frac{1}{n^{R}}\right),
\end{equation}
for an implied constant depending only on $R$ and $v$, with
$\psi_r(v)$ given explicitly in \e{xpo}.
\end{theorem}

From \cite[p. 529]{Olv}, $\ei(n)$ may  be expressed with a contour integral whose path of integration runs along the positive reals while moving above $1$ to avoid the pole:
\begin{equation} \label{ele}
  \ei(n) = -\pi i +\int_0^\infty \frac{e^{n(1-z)}}{1-z} \, dz.
\end{equation}
Make the replacement
\begin{equation*}
  \frac 1{1-z} = 1+z +z^2 + \cdots z^{n+v-1} +  \frac{z^{n+v}}{1-z}
\end{equation*}
in \e{ele} to find
\begin{equation*}
  \ei(n) = -\pi i +e^n \sum_{j=0}^{n+v-1} \frac{j!}{n^{j+1}}- \int_0^\infty e^{n \cdot p(z)} \frac{z^v}{z-1} \, dz,
\end{equation*}
for $p(z)=1-z+\log z$. Hence
\begin{equation} \label{mc}
  \frac{e^n}{n^{n+v}} (n+v)! \Psi_n(v) =  -n\pi i - n\int_0^\infty e^{n  \cdot p(z)} \frac{z^v}{z-1} \, dz.
\end{equation}
We would like to reuse our work in section \ref{w=1} to find the asymptotics of the integral in \e{mc}. As well as having a saddle-point at $z=1$, the integrand  also has a simple pole there and so Theorem \ref{il} cannot be used. Perron in \cite{Pe17} covered the case we need and we quote a version of his result in Theorem 6.3 of \cite{OSper} next, (though it is slightly more general than required). Note that $R_p$  depends only on the holomorphic function $p(z)$ and $z_0$; it can be any positive number that is sufficiently small.

\begin{theorem}  \label{m6} {\rm (Perron's method for an integrand containing a factor $(z-z_0)^{a-1}$  for arbitrary $a \in \C$.)}
Suppose Assumptions \ref{ma0} hold, though with the following change to
 the  contour $\cc$. Starting at $z_1$ it runs to the point $z'_1$ which is a distance $R_p$ from $z_0$ and  on the bisecting line with angle $\theta_{k_1}$. Then the contour circles $z_0$ to arrive at the point $z'_2$ which is a distance $R_p$ from $z_0$ and  on the bisecting line with angle $\theta_{k_2}$. Finally, the contour ends at $z_2$. The integers $k_1$ and $k_2$ keep track of how $\cc$ rotates about $z_0$ between $z'_1$ and $z'_2$; the angle of rotation is  $2\pi(k_2-k_1)/\mu$.

Suppose  that
$
    \Re(p(z))<\Re(p(z_0))
$
for all $z$ in the segments of $\cc$ between $z_1$ and $z'_1$ and between $z'_2$ and $z_2$ (including endpoints).
Let $a \in \C$. For $z\in \cc$, the branch of $(z-z_0)^{a-1}$ is specified by requiring
\begin{equation} \label{thhxm6a}
    (z'_1-z_0)^{a-1} = |z'_1-z_0|^{a-1}\cdot e^{i\theta_{k_1}(a-1)}
\end{equation}
when $z=z'_1$ and by continuity at the other points of $\cc$.
Then for any $S \in \Z_{\gqs 0}$,
\begin{multline} \label{wimm5}
    \int_\cc e^{n \cdot p(z)} (z-z_0)^{a-1} q(z) \, dz \\
    = e^{n \cdot p(z_0)} \left(\sum_{s=0}^{S-1} \G\left(\frac{s+a}{\mu}\right) \frac{\alpha_s  \left( e^{2\pi i {k_2}(s+a)/\mu}- e^{2\pi i {k_1}(s+a)/\mu}\right)}{n^{(s+a)/\mu}} + O\left(\frac{K_q}{n^{(S+\Re(a))/\mu}} \right)  \right)
\end{multline}
where the implied constant in \eqref{wimm5} is independent of $n$ and $q$. The numbers $\alpha_s$
 are given by \eqref{hjw}, depending on $a$ now.
 If $(s+a)/\mu \in \Z_{\lqs 0}$ then
 \begin{equation*}
    \G((s+a)/\mu) \left( e^{2\pi i {k_2}(s+a)/\mu}- e^{2\pi i {k_1}(s+a)/\mu}\right)
 \end{equation*}
  in \eqref{wimm5} is not defined and  must be replaced by $2\pi i (k_2-k_1)(-1)^{(s+a)/\mu}/|(s+a)/\mu|!$.
  \end{theorem}

  We may apply Theorem \ref{m6} to the integral in \e{mc} taking $z_1=1/2$, $z_0=1$ and $z_2=3/2$, since the remaining parts are exponentially small by the work in Lemma \ref{xm}. Then use \e{mu2},  $a=0$, $k_1=1$ and $k_2=0$, to obtain
\begin{equation} \label{mc2}
  \int_0^\infty e^{n p(z)} \frac{z^v}{z-1} \, dz = -\pi i +
  \sum_{s=1}^{S-1}  \G\left(\frac{s}{2}\right) \frac{\alpha_s \cdot (1- (-1)^{s})}{n^{s/2}} + O\left(\frac{1}{n^{S/2}} \right).
\end{equation}
Using \e{mc2} in \e{mc} and simplifying $\alpha_s$ in \e{hjw}  shows the next result.

\begin{prop} \label{xte2}
As  $n \to \infty$,
\begin{equation}\label{abei}
  \frac{e^n}{n^{n+v}} \G(n+v+1) \Psi_n(v) = \sqrt{2\pi n}\left(\tau_0(v)+\frac{\tau_1(v)}{n}+\frac{\tau_2(v)}{n^2}+ \cdots + \frac{\tau_{R-1}(v)}{n^{R-1}}+
  O\left( \frac{1}{n^R}\right)\right),
\end{equation}
for an implied constant depending only on $R\in \Z_{\gqs 1}$ and $v\in \Z$, with
\begin{equation}\label{tav}
  \tau_r(v) :=   \sum_{m=0}^{2r+1} (-1)^{m+1}\binom{v}{2r+1-m}\sum_{k=0}^{m}  \frac{(2r+2k-1)!!}{(-1)^k k!} \dm_{m,k}\left( \frac{1}{3}, \frac{1}{4}, \frac{1}{5}, \dots \right).
\end{equation}
\end{prop}

\begin{proof}[Proof of Theorem \ref{xte}]
Combining Propositions \ref{xpg} and \ref{xte2} produces
\begin{equation*}
  \Psi_n(v) = \left(\sum_{r=0}^{R-1}\frac{\tau_r(v)}{n^r}+
  O\left( \frac{1}{n^R}\right)\right)\Big/\left(1+\sum_{r=1}^{R-1}\frac{\g_r(v)}{n^r}+
  O\left( \frac{1}{n^R}\right)\right).
\end{equation*}
Then \e{trw} follows and $\psi_r(v)$ may be expressed in terms of the $\tau_r(v)$ and $\g_r(v)$ coefficients. Using \cite[Prop. 3.2]{odm}, for example, to find the multiplicative inverse of the series involving $\g_r(v)$ yields
\begin{equation}\label{xpo}
  \psi_r(v) = \sum_{m=0}^r \tau_{r-m}(v) \sum_{k=0}^m (-1)^k \dm_{m,k}\left(\g_1(v), \g_2(v), \dots \right).
\end{equation}
\end{proof}

A computation now finds for example, with $v$ any fixed integer as $n \to \infty$,
\begin{multline} \label{pex}
  \Psi_n(v)  = -\frac{1}3-v+ \left(\frac{4}{135}+\frac{v(v+1)^2}{3}\right)\frac 1{n}\\
   +
   \left(\frac{8}{2835}-\frac{v(9v^4+45v^3+75v^2+47v+8)}{135}\right)\frac{1}{n^2}+O\left( \frac{1}{n^3}\right).
\end{multline}
The expansion  \e{hrb} of $\theta_n(v)$  looks similar to \e{pex} and, in particular, their constant terms $\psi_r=\psi_r(0)$ and $\rho_r$ seem to agree up to an alternating sign.

\begin{conj} \label{ky}
For all $r \gqs 0$ we have $\psi_r = (-1)^{r+1}\rho_r$.
\end{conj}

We have confirmed this relation for $r \lqs 100$ and hope to pursue it in a followup work.

{\small \bibliography{ram-bib} }

{\small 
\vskip 5mm
\noindent
\textsc{Dept. of Math, The CUNY Graduate Center, 365 Fifth Avenue, New York, NY 10016-4309, U.S.A.}

\noindent
{\em E-mail address:} \texttt{cosullivan@gc.cuny.edu}
}

\end{document}